\documentclass[12pt]{amsart}
\usepackage{amssymb}
\usepackage{amsfonts}
\usepackage{amsmath}

\newtheorem{theorem}{Theorem}

\newtheorem{corollary}[theorem]{Corollary}

\newtheorem{lemma}[theorem]{Lemma}

\newcommand{\R}{\mathbb{R}}

\begin{document}
\title[The Dirichlet problem on the slab]{The Dirichlet problem for the slab
with entire data and a difference equation for harmonic functions}

\author[D. Khavinson]{Dmitry Khavinson}
\address{
Dept. of Mathematics and Statistics\\
University of South Florida\\
Tampa, FL, USA}
\email{dkhavins@usf.edu}

\author[E. Lundberg]{Erik Lundberg}
\address{
Dept. of Mathematical Sciences\\
Florida Atlantic University\\
Boca Raton, FL, USA}
\email{elundber@fau.edu}

\author[H. Render]{Hermann Render}
\address{
School of Mathematics and Statistics\\
University College Dublin\\
Dublin, Ireland}
\email{hermann.render@ucd.ie}

\begin{abstract}
It is shown that the Dirichlet problem for the slab $(a,b) \times \mathbb{R}%
^{d}$ with entire boundary data has an entire solution. The proof is based
on a generalized Schwarz reflection principle. Moreover, it is shown that
 for a given entire harmonic function $g$
the inhomogeneous difference equation $h\left( t+1,y\right) -h\left(t,y\right) =g\left( t,y\right)$ 
has an entire harmonic solution $h$.

\smallskip

\noindent Math subject classification (2010): 31B20
\end{abstract}

\maketitle

\section{Introduction}

It is well known that the Dirichlet problem for \emph{unbounded} domains
differs in many respects from the case of bounded domains due to the
non-uniqueness of solutions. An excellent discussion of the Dirichlet
problem for general unbounded domains can be found in \cite{Gard93}.

Maybe the simplest example of this kind is the Dirichlet problem for the
strip $\left( a,b\right) \times \mathbb{R}$ which has been considered by
Widder in \cite{Widd60}, see also \cite{Dura03}. A discussion of the
Dirichlet problem for half-spaces can be found in \cite{Gard81}, \cite%
{SiTa96}, and for a cylinder in \cite{Miya96}.

In this paper we are concerned with the harmonic extendibility of the
solution of the Dirichlet problem for entire data on the slab (see \cite%
{Braw71}) 
\begin{equation*}
S_{a,b}:=\left( a,b\right) \times \mathbb{R}^{d}.
\end{equation*}

We say that a function $f:\mathbb{R}^{d}\rightarrow \mathbb{C}$ is entire if
there exists an analytic function $F:\mathbb{C}^{d}\rightarrow \mathbb{C}$
such that $F\left( x\right) =f\left( x\right)$ for all $x\in \mathbb{R}^{d}$%
. Thus, an entire function $f:\mathbb{R}^{d}\rightarrow \mathbb{C}$ is real
analytic, and it possesses an everywhere convergent power series expansion.
It is well known that every harmonic function $h:\mathbb{R}^{d}\rightarrow 
\mathbb{C}$ is entire.

Our first main result in this paper is the following:

\begin{theorem}
\label{thm:ext1}Let $h$ be a solution of the Dirichlet problem for the slab $%
S_{a,b}$ for entire data $f_{0}$, $f_{1}:\mathbb{R}^{d}\rightarrow \ \mathbb{%
C},$ i.e. $h$ is harmonic on $S$ and $\lim_{t\rightarrow a}h\left(
t,y\right) =f_{0}\left( y\right) $ and $\lim_{t\rightarrow b}h\left(
t,y\right) =f_{1}\left( y\right) .$ Then $h$ extends to all of $\mathbb{R}%
^{d+1}$ as a harmonic function.
\end{theorem}

A similar result holds for the Dirichlet problem for the ellipsoid: H.S.
Shapiro and the first author have established in \cite{KhSh92} that for each
entire data function there exists a solution of the Dirichlet problem which
extends to a harmonic function defined on $\mathbb{R}^{d}$, see also \cite%
{Armi04} for further extensions. 
For the case of a cylinder with ellipsoidal base it is not yet
known whether for any entire data function there exists an entire harmonic
solution, see \cite{KLR16} and  \cite{GaRe16} for partial results. We refer
the reader to a discussion in \cite{EKS05},  \cite{KhLu09}, \cite{Khav91}
and \cite{Rend08} regarding the question of which domains $\Omega $ allow
entire extensions for entire data.

From Theorem \ref{thm:ext1} we shall derive our second main result:

\begin{theorem}
If $g:\mathbb{R}\times \mathbb{R}^{d}\rightarrow \mathbb{C}$ is harmonic
then the difference equation 
\begin{equation}  \label{Diff}
h\left( t+1,y\right) -h\left( t,y\right) =g\left( t,y\right)
\end{equation}
has a harmonic solution $h:\mathbb{R}\times \mathbb{R}^{d}\rightarrow 
\mathbb{C}$.
\end{theorem}

Let us recall now some notations and definitions. A function $f:\Omega
\rightarrow \mathbb{C}$ defined on a domain $\Omega $ in the Euclidean space 
$\mathbb{R}^{d}$ is called \emph{harmonic }if $f$ is twice continuously
differentiable and $\Delta f\left( x\right) =0$ for all $x\in \Omega $ where 
\begin{equation*}
\Delta =\frac{\partial ^{2}}{\partial x_{1}^{2}}+...+\frac{\partial ^{2}}{%
\partial x_{d}^{2}}
\end{equation*}%
is the Laplace operator. We also write $\Delta _{x}$ instead of $\Delta $ to
indicate the variables for differentiation. We say that a function $g:%
\mathbb{R}\times \mathbb{R}^{d}\rightarrow \mathbb{C}$ is even (odd,
respectively) at $t_{0}$ if 
\begin{equation*}
g\left( t_{0}+t,y\right) =g\left( t_{0}-t,y\right),
\end{equation*}
and $g\left( t_{0}+t,y\right) =-g\left( t_{0}-t,y\right)$, respectively, for
all $t\in \mathbb{R}$ and $y\in \mathbb{R}^{d}$.

\smallskip

\noindent {\bf Acknowledgement.}
This work was initiated at the conference Dynamical Systems and Complex Analysis VII, May 2015, in Naharia, Israel.
The first two authors gratefully acknowledge NSF support for the conference (grant DMS -- 1464939).

\section{The Dirichlet problem on the slab with entire data}

Suppose  $h:\left[ a,b\right] \times \mathbb{R}^{d}\rightarrow \mathbb{C}$ is continuous and harmonic in the open slab 
$\left( a,b\right) \times \mathbb{R}^{d}$ such that $h\left( a,y\right)=h\left( b,y\right) =0$ for all $y\in \mathbb{R}^{d}.$ 
Then it is a well known consequence of the Schwarz reflection principle 
that $h$ extends to a harmonic function on $\mathbb{R}^{d+1}$ which is periodic in the variable $t$
with period $2\left( b-a\right) ,$ i.e. 
\begin{equation*}
h\left( t+2\left( b-a\right) ,y\right) =h\left( t,y\right) .
\end{equation*}

In order to obtain a similar result with arbitrary entire boundary data, 
we shall need the following extension of the Schwarz reflection principle:

\begin{theorem}
\label{ThmSchwarz}Suppose that $\Omega $ is a domain in $\mathbb{R}^{d+1}$
such that for each $x=\left( x_{1},...,x_{d+1}\right) \in \Omega $ the
vector $\widetilde{x}=\left( -x_{1},x_{2},...,x_{d+1}\right) \in \Omega ,$
and let $\Omega _{+},\Omega _{0},\Omega _{-}$ denote the sets of points $%
x\in \Omega $ for which $x_{1}$ is positive, zero, and negative (respectively). 
Suppose that $y\longmapsto F\left( y\right) $ for $y=\left( x_{2},...,x_{d+1}\right) \in \mathbb{R}^{d}$ 
is an entire function, and assume that $h$ is harmonic on $\Omega _{-}$ such that for all $y \in \Omega_0$
we have $h\left( x\right) \rightarrow F\left( y\right)$ as $x\rightarrow y\in \Omega _{0}.$ Then $h$ has a harmonic extension to $\Omega.$
\end{theorem}

\begin{proof}
By the Cauchy-Kovalevskaya Theorem applied to the Laplace operator (see \cite%
[p. 80, Example 11.2]{Khav96}), there is a unique entire function $H$ such
that $H\left( 0,y\right) =F\left( y\right) $ and $\frac{\partial }{\partial x}H(0,y)=0$ for all $y\in \mathbb{R}^{d}$.
Moreover, from the uniqueness part of the Cauchy-Kovalevskaya Theorem it follows that $H$ is even at $t_{0}=0$, 
since $H(-t,y)$ solves the same Cauchy problem as $H(t,y)$. 
Consider the function%
\begin{equation*}
f\left( t,y\right) :=h\left( t,y\right) -H\left( t,y\right)
\end{equation*}
for $\left( t,y\right) \in \Omega _{-}.$ 
Then for each $y\in \R^d$ we have $f\left( t,y\right)\rightarrow 0$ as $t\rightarrow 0,$ and by the Schwarz reflection
principle (see \cite[p. 8]{ArGa01}) $f$ extends to a harmonic function $%
\widetilde{f}$ on $\Omega $ by the formula 
\begin{equation*}
\widetilde{f}\left( t,y\right) =-f\left( -t,y\right)
\end{equation*}%
for all $\left( t,y\right) \in \Omega _{+}.$ Then 
\begin{equation*}
\widetilde{h}\left( t,y\right) :=\widetilde{f}\left( t,y\right) +H\left(
t,y\right)
\end{equation*}
is a harmonic extension of $h$ from $\Omega _{-}$ to $\Omega$, and for $t>0$ we have 
\begin{equation}\label{eq:last}
\widetilde{h}\left( t,y\right) =\widetilde{f}\left( t,y\right) +H\left(t,y\right) =-f\left( -t,y\right) +H\left( t,y\right) =-h\left(-t,y\right) +2H\left( t,y\right) .
\end{equation}
\end{proof}

Now we obtain our first main result:

\begin{theorem}
\label{thm:functeq} Assume that $h\in C\left( \left[ a,b\right] \times 
\mathbb{R}^{d}\right) $ is harmonic in the slab $\left( a,b\right) \times 
\mathbb{R}^{d}$ such that $y\longmapsto h\left( a,y\right) $ and $%
y\longmapsto $ $h\left( b,y\right) $ are entire. Then there exists a
harmonic extension $\widetilde{h}:\mathbb{R}^{d+1}\rightarrow \mathbb{C}.$
\end{theorem}

\begin{proof}
The following provides an inductive step:

\noindent \textbf{Claim.} There is an extension $\widetilde{h}\in C\left( %
\left[ a,2b-a\right] \times \mathbb{R}^{d}\right) $ of $h$ which is harmonic
in $\left( a,2b-a\right) \times \mathbb{R}^{d}$ such that $y\longmapsto $ $%
\widetilde{h}\left( 2b-a,y\right) $ is entire.

Using the Claim and induction, one obtains a harmonic extension on $\left(
a,b+n\left( b-a\right) \right) \times \mathbb{R}^{d}$ for each natural
number $n$ such that $y\longmapsto $ $\widetilde{h} \left( a+n\left(
b-a\right) ,y\right) $ is entire. Similarly, there is a harmonic extension
on $\left( a-n\left( b-a\right) ,b\right) \times \mathbb{R}^{d}$ for each
natural number $n,$ and the proof is complete.

In order to establish the Claim, we may assume that $a<b=0.$ 
Theorem \ref{ThmSchwarz} provides an extension 
$\widetilde{h}\left( t,y\right)$ of $h(t,y)$ to $\left[a,-a\right] \times \mathbb{R}^{d}$.
Moreover, from the proof of Theorem \ref{ThmSchwarz}, $\widetilde{h}\left( t,y\right)$ is given by Equation (\ref{eq:last})
$$\widetilde{h}\left( t,y\right) = -h\left(-t,y\right) +2H\left( t,y\right),$$
where $H$ is an entire harmonic function.
This implies that the restriction $\widetilde{h}\left(-a,y\right)$ is entire since $h\left(a,y\right)$ is assumed entire,
and the result then follows.
\end{proof}

\begin{corollary}
\label{CorSlab}Let $f_{0},f_{1}:\mathbb{R}^{d}\rightarrow \mathbb{C}$ be
entire functions. Then any solution $h$ for the Dirichlet problem for the
slab $\left( a,b\right) \times \mathbb{R}^{d}$ with $h\left( a,y\right)
=f_{0}\left( y\right) $ and $h\left( b,y\right) =f_{1}\left( y\right) $
extends to a harmonic function $h:\mathbb{R}^{d+1}\rightarrow \mathbb{C}$.
\end{corollary}

\begin{proof}
It is known that the Dirichlet problem for the slab with continuous data has
a solution $h$, see e.g. \cite{Gard93}. By Theorem \ref{thm:functeq} the
function $h$ has an entire extension.
\end{proof}

\section{The difference equation for harmonic functions}

We recall from complex analysis \cite[p. 407]{BeGa95} that the inhomogeneous difference equation 
\begin{equation}
f\left( z+1\right) -f\left( z\right) = G \left( z\right)  \label{eqcomplex}
\end{equation}%
for a given entire function $G\left( z\right) $ has an entire solution $f\left( z\right)$. 
This is a classical fact, and the solution given in \cite[p. 407]{BeGa95} uses Bernoulli polynomials, 
an idea which goes back to the work of Guichard, Appel, and Hurwitz more than a century ago (see \cite{Ap, Hu}).

Taking the real part of both
sides and recalling that any harmonic function $g(t,y)$ in the plane is the
real part $\Re e \, \{ G(t + i y) \}$ of some entire function $G(z)$, it
follows that the difference equation 
\begin{equation}  \label{eqharmonic}
h\left( t+1,y\right) -h\left( t,y\right) = g\left( t,y\right)
\end{equation}
for a given harmonic function $g$ on $\mathbb{R}^{2}$ has a harmonic
solution $h$ defined on $\mathbb{R}^{2}.$

In this section, we shall generalize this result to all dimensions of the variable $y\in \mathbb{R}^{d}$.
Our approach is based on solving the Dirichlet problem for the slab $\left[0,1/2\right] \times \mathbb{R}^{d}$,
and thus we do not need special functions as in the above-mentioned classical studies.

It is a remarkable fact that equation (\ref{eqcomplex}) can be solved for meromorphic functions as well. 
It would be interesting to extend our
results to include the difference equation (\ref{eqharmonic}) for $g$ with
singularities (say of a controlled type).

As an intermediate step toward solving the difference equation (\ref%
{eqharmonic}), we provide a solution in the case when $g\left( t,y\right)$
is even:

\begin{theorem}
\label{ThmDE1}Let $g:\mathbb{R}\times \mathbb{R}^{d}\rightarrow \mathbb{C}$
be harmonic and even. Then any solution $h\left( t,y\right) $ of the
Dirichlet problem for the slab $\left[ 0,1/2\right] \times \mathbb{R}^{d}$
with data
\begin{equation*}
h\left( 0,y\right) =-\frac{1}{2}g\left( 0,y\right) \text{ and }h\left( \frac{%
1}{2},y\right) =0
\end{equation*}%
for all $y\in \mathbb{R}^{d}$ induces an entire harmonic solution of the
difference equation 
\begin{equation*}
h\left( t+1,y\right) -h\left( t,y\right) =g\left( t,y\right) .
\end{equation*}
\end{theorem}

\begin{proof}
By Corollary \ref{CorSlab} there exists an entire harmonic function $h\left(
t,y\right) $ such that $h\left( 0,y\right) =-\frac{1}{2}g\left( 0,y\right) $
and $h\left( \frac{1}{2},y\right) =0.$ The last equation and the Schwarz
reflection principle shows that 
\begin{equation}
h\left( \frac{1}{2}+t,y\right) =-h\left( \frac{1}{2}-t,y\right) .
\label{eqh1}
\end{equation}%
Inserting $t=\frac{1}{2}$ in equation (\ref{eqh1}) gives $h\left( 1,y\right)
=-h\left( 0,y\right) =\frac{1}{2}g\left( 0,y\right) .$ Now we consider the
harmonic function 
\begin{equation*}
F\left( t,y\right) =h\left( t,y\right) -\frac{1}{2}g\left( t-1,y\right) .
\end{equation*}%
Then $F\left( 1,y\right) =0,$ and by Schwarz's reflection principle, $%
F\left( 1+t,y\right) =-F\left( 1-t,y\right) $ for $y\in $ $\mathbb{R}^{d}.$
Then 
\begin{eqnarray*}
h\left( 1+t,y\right) &=&F\left( 1+t,y\right) +\frac{1}{2}g\left( t,y\right)
=-F\left( 1-t,y\right) +\frac{1}{2}g\left( t,y\right) \\
&=&-h\left( 1-t,y\right) +\frac{1}{2}g\left( -t,y\right) +\frac{1}{2}g\left(
t,y\right)
\end{eqnarray*}

Since $g$ is even we have $\frac{1}{2}g\left( -t,y\right) +\frac{1}{2}%
g\left( t,y\right) =g\left( t,y\right) $ and 
\begin{equation*}
h\left( 1-t,y\right) =h\left( \frac{1}{2}+\frac{1}{2}-t,y\right) =-h\left( 
\frac{1}{2}-\left( \frac{1}{2}-t\right) ,y\right) =-h\left( t,y\right) .
\end{equation*}%
It follows that $h\left( 1+t,y\right) =h\left( t,y\right) +g\left(
t,y\right) .$
\end{proof}

The next result is surely a part of mathematical folklore; we include an
elementary proof in order to keep the paper self-contained.

\begin{lemma}
\label{LemmaDE2} Let $g\left( t,y\right) $ be an entire harmonic function.
Then there exists an entire harmonic function $u\left( t,y\right) $ such
that 
\begin{equation*}
\frac{\partial }{\partial t}u\left( t,y\right) =g\left( t,y\right) .
\end{equation*}%
If $g\left( t,y\right) $ is odd then $u\left( t,y\right) $ can be chosen to
be even.
\end{lemma}

\begin{proof}
Define $h\left( t,y\right) :=\int_{0}^{t}g\left( \tau ,y\right) d\tau .$
Then $\frac{\partial }{\partial t}h\left( t,y\right) =g\left( t,y\right) $
and 
\begin{equation*}
\Delta _{y}\frac{\partial }{\partial t}h\left( t,y\right) =\Delta
_{y}g\left( t,y\right) =-\frac{\partial ^{2}}{\partial t^{2}}g\left(
t,y\right).
\end{equation*}
We conclude that $\frac{\partial }{\partial t}\left( \Delta
_{y}h\left(t,y\right) +\frac{\partial }{\partial t}g\left( t,y\right)
\right) =0$, and it follows that 
\begin{equation*}
f\left( y\right) :=\Delta _{y}h\left( t,y\right) +\frac{\partial }{\partial t%
}g\left( t,y\right)
\end{equation*}
only depends on $y$ and not on $t.$ Obviously for any entire function $%
f\left( y\right)$ there exists an entire function $G\left(y\right)$ such
that $\Delta _{y} G\left( y\right) =f\left( y\right)$. Then 
\begin{equation*}
u\left( t,y\right) :=h\left( t,y\right) -G\left( y\right)
\end{equation*}
is a solution of the equation $\frac{\partial }{\partial t}u\left(t,y\right)
=g\left( t,y\right)$, and $u\left( t,y\right) $ is harmonic, since we have 
\begin{eqnarray*}
\Delta _{t,y}u\left( t,y\right) &=&\frac{\partial ^{2}}{\partial t^{2}}
h\left( t,y\right) +\Delta _{y}h\left( t,y\right) -\Delta _{y}
G\left(y\right) \\
&=&\frac{\partial ^{2}}{\partial t^{2}}h\left( t,y\right) +\Delta
_{y}h\left( t,y\right) -\Delta _{y}h\left( t,y\right) -\frac{\partial }{%
\partial t}g\left( t,y\right) =0.
\end{eqnarray*}
If $g(t,y)$ is odd, then $h(t,y)$ and hence $u(t,y)$ are both even.
\end{proof}

Now we are able to prove our second main result:

\begin{theorem}
\label{thm:harmdiff} If $g:\mathbb{R}\times \mathbb{R}^{d}\rightarrow 
\mathbb{C}$ is harmonic then the difference equation $h\left( t+1,y\right)
-h\left( t,y\right) =g\left( t,y\right) $ has a harmonic solution $h:\mathbb{%
R}\times \mathbb{R}^{d}\rightarrow \mathbb{C}$.
\end{theorem}

\begin{proof}
Write the harmonic function $g$ as a sum $g_{0}+g_{e}$, where $g_{0}$ is odd
and $g_{e}$ is even. For right hand side $g_{e}$, there exists a solution $%
h_{e}\left(t,y\right) $ by Theorem \ref{ThmDE1}, and so it suffices to solve
the difference equation with right hand side $g_0$. By Lemma \ref{LemmaDE2}
there exists an even harmonic function $u\left( t,y\right)$ such that 
\begin{equation*}
\frac{\partial }{\partial t}u\left( t,y\right) =g_{0}\left( t,y\right).
\end{equation*}
By Theorem \ref{ThmDE1}, there exists a harmonic entire function $%
H\left(t,y\right) $ such that 
\begin{equation*}
H\left( t+1,y\right) -H\left( t,y\right) =u\left( t,y\right).
\end{equation*}
Differentiating $H(t,y)$ with respect to $t$, we thus find the solution of
the difference equation with right hand side $g_{0}\left(t,y\right)$.
\end{proof}

Finally, although the solution to the difference equation is far from unique,
we are able to prove the following result.

\begin{theorem}
Let $g:\mathbb{R}\times \mathbb{R}^{d}\rightarrow \mathbb{C}$ be a harmonic
function and let $h_{j}:\mathbb{R}\times \mathbb{R}^{d}\rightarrow \mathbb{C}
$ for $j=1,2$ be harmonic solutions of the difference equation 
\begin{equation}\label{eq:diff}
h_{j}\left( t+1,y\right) -h_{j}\left( t,y\right) =g\left( t,y\right) .
\end{equation}
Assume that we have the estimate 
\begin{equation}
\left\vert h_{1}(t,y)-h_{2}(t,y)\right\vert =o\left( \left\vert y\right\vert
^{\left( 1-d\right) /2}e^{2\pi \left\vert y\right\vert }\right) ,
\label{eqEst1}
\end{equation}%
as $\left\vert y\right\vert \rightarrow \infty $ (uniformly in $t$). Then
there exists a harmonic function $r:\mathbb{R}^{d}\rightarrow \mathbb{C}$
such that 
\begin{equation*}
h_{1}\left( t,y\right) =h_{2}\left( t,y\right) +r\left( y\right)
\end{equation*}%
for all $y\in \mathbb{R}^{d}$ and $t\in \mathbb{R}$.
\end{theorem}

\begin{proof}
Define $h\left( t,y\right) :=h_{1}\left( t,y\right) -h_{2}\left( t,y\right) $.
Then $h$ is periodic in $t$ with period $1$,
since $h_1$ and $h_2$ each solve the difference equation (\ref{eq:diff}).
Let us also define $H\left( t,y\right) :=H_{y}\left( t\right) :=h\left( t/(2\pi) ,y/(2\pi) \right) $ for $y\in \mathbb{R}^{d},$ $t\in \mathbb{R}.$ 
Then $H_{y}$ is a $2\pi $-periodic function in $t$,
so it has a Fourier series $\sum_{k=-\infty}^{\infty }a_{k}\left( y\right) e^{ikt}.$ 
Applying the Laplace operator to
the Fourier coefficients 
\begin{equation}
a_{k}\left( y\right) =\frac{1}{2\pi }\int_{0}^{2\pi }H_{y}\left( t\right)
e^{ikt}dt,  \label{eqFourier}
\end{equation}%
we have 
\begin{equation*}
\Delta _{y}a_{k}\left( y\right) =\frac{1}{2\pi }\int_{0}^{2\pi }\Delta
_{y}H\left( t,y\right) e^{ikt}dt.
\end{equation*}

Since $H\left( t,y\right) $ is harmonic we know that 
\begin{equation*}
\Delta _{y}H\left( t,y\right) =-\frac{\partial ^{2}}{\partial t^{2}}H\left(
t,y\right) .
\end{equation*}%
Integration by parts then shows that 
\begin{equation*}
\Delta _{y}a_{k}\left( y\right) =-\frac{1}{2\pi }\int_{0}^{2\pi }\frac{%
\partial ^{2}H}{\partial t^{2}}\left( t,y\right) \cdot
e^{ikt}dt=k^{2}a_{k}\left( y\right) .
\end{equation*}%
Hence $a_{k}\left( y\right) $ is a solution of the Helmholtz equation $%
\Delta _{y}a_{k}=k^{2}a_{k}.$ In view of (\ref{eqFourier}) the estimate (\ref%
{eqEst1}) carries over to $a_{k}(y)$.  Since $k^{2}>0$, a classical
result \cite[p. 228]{Frie56}, going back to work of I. Vekua and F. Rellich
in the 1940's, yields that $a_{k}=0$ for all $k\neq 0$.
\end{proof}


\begin{thebibliography}{99}

\bibitem{Ap} 
P. Appell, \emph{Sur les fonctions p\'eriodiques de deux variables}, J. Math. Pures et Appl. {\bf 7} (1891), 157--219.

\bibitem{ArGa01} D. H. Armitage, S. J. Gardiner, \emph{Classical Potential
Theory}, Springer, London 2001.

\bibitem{Armi04} D. Armitage, \emph{The Dirichlet problem when the boundary
function is entire,} J. Math. Anal. \textbf{291} (2004), 565--577.

\bibitem{BeGa95} C. A. Berenstein, R. Gay, \emph{Complex Analysis and
Special Topics in Harmonic Analysis,} Springer Verlag, New York, 1995.

\bibitem{Braw71} F.T. Brawn, \emph{The Poisson integral and harmonic
majorization in} $R^{n}\times \left] 0,1\right[ ,$ J. London Math. Soc. {\bf 3} (1971), 747--760.

\bibitem{Dura03} W. Durand, \emph{On some boundary value problems on a strip
in the complex plane,} Reports on Mathematical Physics, {\bf 52} (2003), 1--23.

\bibitem{EKS05} P. Ebenfelt, D. Khavinson, H.S. Shapiro, \emph{Algebraic
Aspects of the Dirichlet problem,} Operator Theory: Advances and
Applications \textbf{156} (2005), 151--172.

\bibitem{Frie56} A. Friedman, \emph{On }$n$\emph{-metaharmonic functions and
functions of infinite order, }Proc. Amer. Math. Soc. {\bf 8} (1957), 223--229.

\bibitem{Gard81} S.J. Gardiner, \emph{The Dirichlet and Neumann problems for
harmonic functions in half spaces,} J. London Math. Soc. {\bf 24} (1981), 502--512.

\bibitem{Gard93} S.J. Gardiner, \emph{The Dirichlet problem with non-compact
boundary,} Math. Z. {\bf 213} (1993), 163--170.

\bibitem{GaRe16} S.J. Gardiner, H. Render, \emph{Harmonic functions which
vanish on a cylindrical surface,} J. Math. Anal. Appl. {\bf 433} (2016),
1870--1882. 

\bibitem{Hu}
A. Hurwitz, \emph{Sur l'int\'egrale finie d'une fonction enti\`ere}, 
Acta Math. {\bf 20} (1897), 285-312.

\bibitem{Khav91} D. Khavinson, \emph{Singularities of harmonic functions in }%
$C^{n}$\emph{,} Proc. Symp. Pure and Applied Math., American Mathematical
Society, Providence, RI, \textbf{52} (1991), Part 3, 207-217.

\bibitem{Khav96} D. Khavinson, \emph{Holomorphic Partial Differential
Equations and Classical Potential Theory,} Imprecan, S.L., Universidad de La
Laguna, Tenerife, 1996.

\bibitem{Khav97} D. Khavinson, \emph{Cauchy's problem for harmonic functions
with entire data on a sphere,} Can. Math. Bull. \textbf{40} (1997), 60--66.

\bibitem{KhLu09} D. Khavinson, E. Lundberg, \emph{The search for
singularities of solutions to the Dirichlet problem: recent developments,}
CRM Proceedings and Lecture Notes, {\bf 51} (2010), 121--132.

\bibitem{KhSh92} D. Khavinson, H.S. Shapiro, \emph{Dirichlet's problem when
the data is an entire function}, Bull. London Math. Soc. {\bf 24} (1992), 456--468.

\bibitem{KLR16} D. Khavinson, E. Lundberg, H. Render, \emph{Dirichlet's
problem with entire data posed on an ellipsoidal cylinder}, Preprint.

\bibitem{Miya96} I. Miyamoto, \emph{A type of uniquenes of solutions for the
Dirichlet problem on a cylinder,} T\^{o}hoku Math. J. {\bf 48} (1996), 267-292.

\bibitem{Rend08} H. Render, \emph{Real Bargmann spaces, Fischer
decompositions and sets of uniqueness for polyharmonic functions,} Duke
Math. J. \textbf{142} (2008), 313--351.

\bibitem{SiTa96} D. Siegel, E.O. Talvila, \emph{Uniqueness for the $n$%
-dimensional half space Dirichlet problem,} Pacific J. Math. {\bf 175} (1996),
571--587.

\bibitem{Shap89} H.S. Shapiro, \emph{An algebraic theorem of E. Fischer and
the Holomorphic Goursat Problem,} Bull. London Math. Soc. {\bf 21} (1989),
513--537.

\bibitem{Widd60} D.V. Widder, \emph{Functions harmonic in a strip}, Proc.
Amer. Math. Soc. {\bf 12} (1961), 67--72.
\end{thebibliography}
\end{document}